\documentclass[12pt,article]{amsart}
\usepackage{amsmath}
\usepackage{amssymb}
\newtheorem{corollary}[equation]{Corollary}

\newtheorem{claim}[equation]{\indent{\it Claim}\rm }

\newtheorem{lemma}[equation]{Lemma}
\newtheorem{proposition}[equation]{Proposition}

\newtheorem{theorem}[equation]{Theorem}

\newcommand{\C}{{\mathbb{C}}}

\renewcommand{\P}{{\mathbb{P}}}
\newcommand{\E}{{\mathbb{E}}}

\newcommand{\R}{{\mathbb{R}}}

\newcommand{\rank}{\mathrm{rank}}

\newcommand{\supp}{\mathrm{Supp}\,}

\numberwithin{equation}{section}

\title[Modified defect relation of Gauss maps on annular ends]{Modified defect relation of Gauss maps on annular ends of minimal surfaces for hypersurfaces of projective varieties in subgeneral position} 

\author{Si Duc Quang}

\address{$^1$Department of Mathematics, Hanoi National University of Education\\
136-Xuan Thuy, Cau Giay, Hanoi, Vietnam}
\address{$^2$Thang Long Institute of Mathematics and Applied Sciences\\
Nghiem Xuan Yem, Hoang Mai, Hanoi, Vietnam}
\email{quangsd@hnue.edu.vn}

\begin{document}

\begin{abstract}
Let $A$ be an annular end of a complete minimal surface $S$ in $\R^m$ and let $V$ be a $k$-dimension projective subvariety of $\P^n(\C)\ (n=m-1)$. Let $g$ be the generalized Gauss map of $S$ into $V\subset\P^n(\C)$. In this paper, we establish a modified defect relation of $g$ on the annular end $A$ for $q$ hypersurfaces $\{Q_i\}_{i=1}^q$ of $\P^n(\C)$ in $N$-subgeneral position with respect to $V$. Our result implies that the image $g(A)$ cannot omit all $q$ hypersurfaces $Q_1,\ldots,Q_q$ if $g$ is nondegenerate over $I_d(V)$ and $q>\frac{(2N-k+1)(M+1)(M+2d)}{2d(k+1)}$, where $M=H_V(d)-1$ and $d$ is the least of common multiple of $\deg Q_1,\ldots,\deg Q_q$. As our best knowledge, it is the first time the value distribution of the Gauss map on an annular end of a minimal surfaces with hypersurface targets is studied, in particular the product into sum inequality for holomorphic curves on Riemann surfaces with hypersurfaces targets is presented. This our result has been used to study the unicity of the gauss maps in the recent work of C. Lu and X. Chen \cite{LC}.
\end{abstract}

\maketitle

\def\thefootnote{\empty}
\footnotetext{
\hskip8pt This research is funded by Vietnam National Foundation for Science and Technology Development (NAFOSTED) under grant number 101.02-2021.12.}

%\tableofcontents

\section{Introduction and Main results} 
The value distribution of the Gauss maps of minimal surfaces was initially studied by  Osserman in a series of his papers \cite{O59,O63,O64}. Later on, this problem had been deepen by many authors such as F. Xavier \cite{X81}, H. Fujimoto \cite{Fu83, Fu83b, Fu89, Fu90, Fu91}, M. Ru \cite{Ru91}, X. Mo \cite{M94}, G. Dethloff and P. H. Ha \cite{DH,Ha} and others. In 1988, H. Fujimoto \cite{Fu88} showed that the Gauss map of a complete non-flat minimal surface in $\R^3$ can omit at most four points, and this number four is the best possible number. For the case of higher dimension, in 1991, based on an earlier result of H. Fujimoto, M. Ru \cite{Ru91} showed that the generalized Gauss map of a complete non-flat minimal surface $S$ in $\R^m$ cannot omit more than $\frac{m(m+1)}{2}$ hyperplanes of $\P^{m-1}(\C)$ in general position. In the same year, S. J. Kao \cite{Kao} extended H. Fujimoto's method and showed that the Gauss map of an end of a non-flat complete minimal surface in $\R^3$ that is conformally an annulus $\{z; 0 < 1/r < |z| < r\}$ must also assume every value, with at most four exceptions. Later on, L. Jin and M. Ru in \cite{Ru07} extended the result of Kao to the case of higher dimension by showing that the generalized Gauss map on an annular end of a complete minimal surface $S$ in $\R^{m}$ cannot omit more than $\frac{m(m+1)}{2}$ hyperplanes of $\P^{m-1}(\C)$ in general position. Recently, G. Dethloff, P. H. Ha and P. D. Thoan \cite{DHT} have improved the result of L. Jin and M. Ru by considering the ramification of the generalized Gauss map on an annular end over a family of hyperplanes of $\P^n(\C)$ in subgeneral position. 

In this paper, we will give a modified defect relation for the generalized Gauss map on an annular end of a complete minimal surface in $\R^m$ into a projective subvariety $V\subset\P^n(\C)\ (n=m-1)$ for hypersurfaces of $\P^n(\C)$ in subgeneral position with respect to $V$. Our result extends and generalizes all above mentioned results. In order to state the main result of this paper, we will introduce the notion of modified defect for hypersurfaces in a projective subvariety $V\subset\P^n(\C)$ of a holomorphic curve $f$ from an open Riemann surface into $V$, which is a generalization of the notion of modified defect for hypersurplanes in $\P^n(\C)$ introduced by H. Fujimoto \cite{Fu89,Fu90,Fu91}), as follows.

Let $S$ be an open complete Riemann surface in $\R^m$. By a divisor on $S$ we mean a map $\nu$ of $S$ into $\mathbb Z$ whose support $\supp(\nu) :=\{p\in S; \nu(p)\ne 0\}$ has no accumulation points in $S$. Let $A$ be an annular end on $S$ and let $f$ be a holomorphic curve from $A$ to a $k$-dimension projective subvariety $V$ of $\P^n(\C)$ and let $Q$ be a hypersurface in $\P^n(\C)$ of degree $d$. By $\nu_{Q(f)}$ we denote the pull-back of the divisor $Q$ by $f$. Let $F=(f_0,\ldots,f_n)$ be a global reduced representation of $f$ on $A$. Assume that, the hypersurface $Q$ has a defining polynomial, denoted again by the same notation $Q$ (throughout this paper) if there is no confusion, as follows
$$ Q(x_0,\ldots,x_n)=\sum_{I\in\mathcal T_d}a_Ix^I,$$ 
where $\mathcal T_d=\{(i_0,\ldots,i_n)\in\mathbb Z^{n+1}_+; i_0+\cdots+i_n=d\}$, $a_I\in\C$ are not all zero for $I\in\mathcal T_d$ and $x^I=x_0^{i_0}\ldots x_n^{i_n}$ for each $I=(i_0,\ldots,i_n)$. We set
$$ Q(F)=\sum_{I\in\mathcal T_d}a_If^I,$$
where $f^I=f_0^{i_0}\ldots f_n^{i_n}$ for each $I\in\mathcal T_d$. Throughout this paper, for each given hypersurface $Q$ we assume that $\|Q\|=(\sum_{I\in\mathcal T_d}|a_I|^2)^{1/2}=1$. The $H$-defect truncated to level $m$ for $Q$ of $f$ on $A$ is defined by
$$\delta^{H,m}_{f,A}(Q):=1- \mathrm{inf}\{\eta; \eta \text{ satisfies Condition }(*)_H\}.$$
Here, Condition $(*)_H$ means that there exists a $[-\infty,\infty)$-valued continuous subharmonic function $u\ (\not\equiv -\infty)$ on $A$, harmonic on $A\setminus f^{-1}(H)$ and satisfying conditions:
\begin{itemize}
\item[$(D_1)$] $e^u\le \|F\|^{d\eta}$,
\item [$(D_2)$] for each $\xi\in f^{-1}(H)$ there exists the limit 
$$\underset{z\rightarrow\xi}\lim (u(z)-\min(\nu_{Q(f)}(z),m)\log|z -\xi|) \in[-\infty,\infty).$$
\end{itemize}

Denote by $I(V)$ the ideal of homogeneous polynomials in $\C [x_0,...,x_n]$ defining $V$ and by $H_d$ the vector space of all homogeneous polynomials in $\C [x_0,...,x_n]$ of degree $d$ including the zero polynomial. Define 
$$I_d(V):=\dfrac{H_d}{I(V)\cap H_d}\text{ and }H_V(d):=\dim I_d(V).$$
For an element $D\in H_d,$ we denote by $[D]$ its equivalent class in $I_d(V)$. 

Let $Q_1,...,Q_q\ (q\ge k+1)$ be $q$ hypersurfaces in $\P^n(\C)$. The hypersurfaces $Q_1,...,Q_q$ are said to be in $N$-subgeneral position with respect to $V$ if 
$$ V\cap \left(\bigcap_{j=1}^{N+1}Q_{i_j}\right)=\varnothing \ \forall\ 1\le i_1<\cdots <i_{N+1}\le q.$$
If  $Q_1,...,Q_q$ are in $k$-subgeneral position with respect to $V$ then we say that they are in \textit{general position} with respect to $V.$

For $f$ as above, we say that $f$ is algebraic degenerate over $I_d(V)$ if there is $[Q]\in I_d(V)\setminus \{0\}$ such that $Q(F)\equiv 0$ for some reduced representation $F$ of $f$. Otherwise, we say that $f$ is nondegenerate over $I_d(V)$. It is clear that if $f:A\rightarrow V$ is algebraic nondegenerate, then $f$ is nondegenerate over $I_d(V)$ for every $d\ge 1.$ Our main result is stated as follows.
 
\begin{theorem}\label{1.1} 
Let $A$ be an annular end of a non-flat complete minimal surface $S$ in $\R^m$. Let $V$ be a $k$-dimension projective subvariety of $\P^{n}(\C)\ (n=m-1)$. Let $Q_1,\ldots,Q_q$ be hypersurfaces of $\P^n(\C)$ in $N$-subgeneral position with respect to $V$ and $d$ the least common multiple of $\deg Q_j\ (1\le j\le q).$ Assume that the generalized Gauss map $g:S\rightarrow V\subset\P^n(\C)$ is nondegenerate over $I_d(V)$. Then
$$\sum_{j=1}^q\delta^{H,M}_{g,A}(Q_j)\le \dfrac{(2N-k+1)(M+1)}{k+1}+\dfrac{(2N-k+1)M(M+1)}{2d(k+1)},$$
where $M=H_V(d)-1$. 
\end{theorem}
From Theorem \ref{1.1} we immediately have the following corollary about the ramification of $f$ over a family of hypersurfaces on an annular end.

\begin{corollary}\label{1.2} 
Let $S, A, V,\{Q_j\}_{j=1}^q,d,M$ be as in Theorem \ref{1.1}. Assume that the generalized Gauss map $g:S\rightarrow V\subset\P^n(\C)$ is nondegenerate over $I_d(V)$ and ramified over $Q_j$ on $A$ with multiplicity at least $m_j\ (1\le j\le q)$. Then
$$\sum_{j=1}^q\left(1-\frac{M}{m_j}\right)\le \dfrac{(2N-k+1)(M+1)}{k+1}+\dfrac{(2N-k+1)M(M+1)}{2d(k+1)}.$$
In particular, $g$ cannot omit $Q_1,\ldots,Q_q$ on $A$ if $q>\frac{(2N-k+1)(M+1)}{k+1}+\frac{(2N-k+1)M(M+1)}{2d(k+1)}.$
\end{corollary}
\begin{proof}
Indeed, since $g$ is ramified over $Q_j$ with multiplicity at least $m_j$ on $A$ for each $j$, we have
$$ \delta^{H,M}_{g,S}(Q_j)\ge 1-\dfrac{M}{m_j}\ \ (1\le j\le q).$$
Then by Theorem \ref{1.1}, we have
$$\sum_{j=1}^q\left(1-\frac{M}{m_j}\right)\le \dfrac{(2N-k+1)(M+1)}{k+1}+\dfrac{(2N-k+1)M(M+1)}{2d(k+1)}.$$
The corollary is proved.
\end{proof}
 We note that, in the case of hyperplanes in $\P^n(\C)$, i.e., $d=1$ and $V=\P^n(\C)$, then $M=n$. Therefore, this corollary immediately implies the results mentioned above of  S. J. Kao \cite{Kao}, L. Jin-M. Ru \cite{Ru91} and G. Dethloff - P. H. Ha- P. D. Thoan \cite{DHT}.
\section{Main lemmas}

\begin{lemma}[{cf. \cite[Lemma 3]{QA}}]\label{lem2.3}
Let $V$ be a $k$-dimension projective subvariety of $\P^n(\C)$. Let $Q_1,...,Q_q$ be $q\ (q>2N-k+1)$ hypersurfaces of $\P^n(\C)$ in $N$-subgeneral position with respect to $V$ of the same degree. Then, there are positive rational constants $\omega_i\ (1\le i\le q)$ satisfying the following:

i) $0<\omega_i \le 1\  \forall i\in\{1,...,q\}$,

ii) Setting $\tilde \omega =\max_{j\in Q}\omega_j$, one gets $\sum_{j=1}^{q}\omega_j=\tilde \omega (q-2N+k-1)+k+1.$

iii) $\dfrac{k+1}{2N-k+1}\le \tilde\omega\le\dfrac{k}{N}.$

iv) For each $R\subset \{1,...,q\}$ with $\sharp R = N+1$, then $\sum_{i\in R}\omega_i\le k+1$.

v) Let $E_i\ge 1\ (1\le i \le q)$ be arbitrarily given numbers. For each $R\subset \{1,...,q\}$ with $\sharp R = N+1$,  there is a subset $R^o\subset R$ such that $\sharp R^o=\rank \{Q_i\}_{i\in R^o}=k+1$ and 
$$\prod_{i\in R}E_i^{\omega_i}\le\prod_{i\in R^o}E_i.$$
\end{lemma}

Let $S$ be an open Riemann surface and let $ds^2$ be a pseudo-metric on $S$ which is locally written as $ds^2=\lambda^2|dz|^2$, where $\lambda$ is a nonnegative real-value function with mild singularities and $z$ is a holomorphic local coordinate. The divisor of $ds^2$ is defined by $\nu_{ds}:=\nu_\lambda$ for each local expression $ds^2=\lambda^2|dz|^2$, which is globally well-defined on $S$. We say that $ds^2$ is a continuous pseudo-metric if $\nu_{ds}\ge 0$ everywhere.

The Ricci of  $ds^2$ is defined by 
$$ \mathrm{Ric}_{ds^2}=dd^c\log\lambda^2$$
for each local expression  $ds^2=\lambda^2|dz|^2$. This definition is globally well-defined on $S$.

The continuous pseudo-metric $ds^2$ is said to have strictly negative curvature on $S$ if there is a positive constant $C$ such that
$$ -\mathrm{Ric}_{ds^2}\ge C\Omega_{ds^2}, $$
where $\Omega_{ds^2}=\lambda^2\cdot\dfrac{\sqrt{-1}}{2}\cdot dz\wedge d\bar z$.

Let $f$ be a holomorphic map of $S$ into a projective subvariety $V$ of dimension $k$ of $\P^n(\C)$. Let $d$ be a positive integer and assume that $f$ is nondegenerate over $I_d(V)$. We fix a $\C$-ordered basis $\mathcal V=([v_0],\ldots,[v_M])$ of $I_d(V)$, where $v_i\in H_d$ and $M=H_V(d)-1$.  

Let $F=(f_0,\ldots,f_n)$ be a local reduced representation of $f$ (in a local chat $(U,z)$ of $S$). 
Consider the holomorphic map
$$ F_{\mathcal V}=(v_0(F),\ldots,v_M(F))$$
and 
$$F_{\mathcal V,p}:= F_{\mathcal V}^{(0)} \wedge F_{\mathcal V}^{(1)} \wedge \cdots \wedge F_{\mathcal V}^{(p)} : S\rightarrow \bigwedge_{p+1}\C^{M+1}$$
for $0\le p\le M$, where 
\begin{itemize}
\item $F_{\mathcal V}^{(0)}:=F_{\mathcal V}=(v_0(F),\ldots,v_M(F))$,
\item $F_{\mathcal V}^{(l)}:=\left (v_0(F)^{(l)},\ldots, v_M(F)^{(l)}\right)$ for each $l=0, 1,\ldots , p$,
\item $v_i(F)^{(l)} \ (i =0,\ldots, M)$ is the $l^{th}$- derivatives of $v_i(F)$ taken with respect to $z$.
\end{itemize}
The norm of $F_{\mathcal V,p}$ is given by
$$|F_{\mathcal V,p}|:=\left (\sum_{0\le i_0<i_1<\cdots<i_p\le M}\left |W(v_{i_0}(F),\ldots,v_{i_p}(F))\right|^2\right)^{1/2}, $$
where 
$$W(v_{i_0}(F),\ldots,v_{i_p}(F)):=\det\left (v_{i_j}(F)^{(l)}\right)_{0\le l,j\le p}.$$ 
% We have some fundamental properties of the wronskian of holomorphic function $h_0,h_1,\ldots,h_p$ as follows:
% \begin{itemize}
% \item $W_{\xi}(h_0,\ldots,h_p)=W_z(h_0,\ldots,h_p) (\xi'_z)^{\frac{p(p+1)}{2}}$,
% \item $W_z(hh_0,\ldots,hh_p)=h^{p+1}W_z(h_0,\ldots,h_p)$,
% \item $h_0,\ldots,h_p$ are $\C$-linearly dependent if and only if $W_z(h_0,\ldots,h_p)\equiv 0$.
% \end{itemize}

We use the same notation $\langle,\rangle$ for the canonical hermitian product on $\bigwedge^{l+1}\C^{M+1}\ (0\le l\le M)$. For two vectors $A\in \bigwedge^{k+1}\C^{M+1}\ (0\le k\le M)$ and $B\in\bigwedge^{p+1}\C^{M+1}\ (0\le p\le k)$, there is one and only one vector $C\in\bigwedge^{k-p}\C^{M+1}$ satisfying 
$$ \langle C,D\rangle=\langle A,B\wedge D\rangle\ \forall D\in \bigwedge^{k-p}\C^{M+1}.$$
The vector $C$ is called the interior product of $A$ and $B$, and defined by $A\vee B$.

Now, for a hypersurface $Q$ of degree $d$ in $\P^n(\C)$, we have
$$[Q]=\sum_{i=0}^Ma_i[v_i].$$
Hence, we associate $Q$ with the vector $H=(a_0,\ldots,a_M)\in\C^{M+1}$ and define $F_{\mathcal V,p}(Q)=F_{\mathcal V,p}\vee H$. Then, we may see that
\begin{align*}
F_{\mathcal V,0}(Q)&=a_0v_0(F)+\cdots+a_Mv_M(F)=Q(F),\\ 
|F_{\mathcal V,p}(Q)|&=\left (\sum_{0\le i_1<\cdots<i_p\le M}\sum_{l\ne i_1,\ldots,i_p}a_l\left |W(v_l(F),v_{i_1}(F),\ldots,v_{i_p}(F))\right|^2\right)^{1/2}.
\end{align*}
Finally, for $0\le p\le M$, the $p^{th}$-contact function of $f$ for $Q$ with respect to $\mathcal V$ is defined (not depend on the choice of the local coordinate) by
$$\varphi_{\mathcal V,p}(Q):=\dfrac{|F_{\mathcal V,p}(Q)|^2}{|F_{\mathcal V,p}|^2}.$$

For each $p\ (0\le p\le M-1)$, let $M_p=\binom{M+1}{p+1}-1$ and $\pi_p$ the canonical projection from $\bigwedge^{p+1}\C^{M+1}\sim\C^{M_p+1}$ onto $\P^{M_p}(\C)$. Denote by $\Omega_p$ the pull-back of the Fubini-Study form on $\P^{M_p}(\C)$ by the map $\pi\circ F_{\mathcal V,p}$, i.e., $\Omega_p = dd^c\log |F_{\mathcal V,p}|^2$. 
\begin{proposition}[{cf. \cite[Proposition 2.5.1]{Fu93}}]\label{pro3.1}
Let $S,V,d,\mathcal V$ and $M$ be as above. For each positive number $\epsilon$ there exists a constant $\delta_0(\epsilon)$, depending only on $\epsilon$, such that for any hypersurface $Q$ of degree $d$ in $\P^n(\C)$ and any constant $\delta>\delta_0(\epsilon)$
$$ dd^c\log\dfrac{1}{\log(\delta/\varphi_{\mathcal V,p}(Q))}\ge\dfrac{\varphi_{\mathcal V,p}(Q)}{\varphi_{\mathcal V,p+1}(Q)\log(\delta/\varphi_{\mathcal V,p}(Q))}\Omega_p-\epsilon\Omega_p.$$
\end{proposition}
Actually, \cite[Proposition 2.5.1]{Fu93} states only for hyperplanes. However, by using Serge embedding to embed $\P^n(\C)$ into $\P^M(\C)$, we automatically deduce the above proposition.

\begin{theorem}[{cf. \cite[Proposition 2.5.7]{Fu93}}]\label{thm3.4}
Set $\sigma_p=p(p+1)/2$ for $0\le p\le M+1$ and $\tau_m=\sum_{p=1}^m\sigma_m$. Then, we have
$$ dd^c\log(|F_{\mathcal V,0}|^2\cdots |F_{\mathcal V,M-1}|^2)\ge\dfrac{\tau_M}{\sigma_M}\left(\dfrac{|F_{\mathcal V,0}|^2\cdots |F_{\mathcal V,M}|^2}{|F_{\mathcal V,0}|^{2\sigma_{M+1}}}\right)^{1/\tau_M}dd^c|z|^2. $$
\end{theorem}

\begin{theorem}\label{thm3.5}
With the notations and the assumption in Theorem \ref{thm3.3}, we have
$$ \nu_{\phi}+\sum_{j=1}^q\omega_j\cdot\min\{\nu_{Q_j(f)},M\}\ge 0, $$
where $\phi=\dfrac{|F_{\mathcal V,M}|}{\prod_{j=1}^q|Q_j(F)|^{\omega_j}}$ and $\nu_\phi$ is the divisor of the function $\phi$.
\end{theorem}

\begin{proof}
Fix a point $z\in S$. Since $\{Q_i\}_{i=1}^q$ is in $N$-subgeneral position, $z$ is not zero of more than $N$ functions $Q_i(F)$. Suppose that $z$ is zero of $Q_i(F)$ for every $i=1,\ldots,\ell\ (\ell\le N)$ and $z$ is not zero of $Q_i(F)$ for every $i>\ell$. Put $R=\{1,...,N+1\}.$ Choose $R^1\subset R$ such that 
$\sharp R^1=\rank\{Q_i\}_{i\in R^1}=k+1$ and $R^1$ satisfies Lemma \ref{lem2.3} v) with respect to numbers $\bigl \{e^{\max\{\nu_{Q_i(F)}(z)-M,0\}} \bigl \}_{i=1}^q.$ Then we have
\begin{align*}
\nu_{F_{\mathcal V,M}}(z)&\ge \sum_{i\in R^1}\max\{\nu_{Q_i(f)}(z)-M,0\}\\
&\ge\sum_{i\in R}\omega_i \max\{\nu_{Q_i(f)}(z)-M,0\}. 
\end{align*}
Hence 
\begin{align*}
\sum_{i=1}^q\omega_i\nu_{Q_i(f)}(z)-\nu_{F_{\mathcal V,M}}(z)& =\sum_{i\in R}\omega_i\nu_{Q_i(f)}(z)-\nu_{F_{\mathcal V,M}}(z)\\ 
&\le \sum_{i\in R}\omega_i\nu_{Q_i(f)}(z)-\sum_{i\in R}\omega_i\max\{\nu_{Q_i(f)}(z)-M,0\}\\
&=\sum_{i=1}^q\omega_i\min\{\nu_{Q_i(f)}(z),M\}.
\end{align*}
The theorem is proved.
\end{proof}

\begin{lemma}[{Generalized Schwarz's Lemma \cite{A38}}]\label{lem3.6}  
Let $v$ be a non-negative real-valued continuous subharmonic function on $\Delta (R)=\{z\in\C; \|z\|<R\}$. If $v$ satisfies the inequality $\Delta\log v\ge v^2$ in the sense of distribution, then
$$v(z) \le \dfrac{2R}{R^2-|z|^2}.$$
\end{lemma}

Next, we give the following two key theorems in this paper, Theorems \ref{thm3.2} and \ref{thm3.3}. These theorems have been initially proved in another recent work of the author \cite{Q}. For sake of completeness, we also give their detail proofs here.
\begin{theorem}\label{thm3.2}
Let $S,V,d,\mathcal V$ and $M$ be as above. Let $f:S\rightarrow V\subset\P^n(\C)$ be a holomorphic curve and let $Q_1,\ldots, Q_q$ be hypersurfaces of $\P^n(\C)$ in $N$-subgeneral position with respect to $V$ of the same degree $d$, where $q>\frac{(2N-k+1)(M+1)}{k+1}$. Assume that $f$ is nondegenerate over $I_d(V)$ and have a local reduced representation $F=(f_0,\ldots,f_n)$ on a local holomorphic chat $(U,z)$. Let $\omega_j\ (1\le j\le q)$ be the Nochka weights for these hypersurfaces (defined in Lemma \ref{lem2.3}). For an arbitrarily given $\delta >1$ and $0\le p\le M-1$, we set
$$\Phi_{\mathcal V,jp}:=\dfrac{\varphi_{\mathcal V,p+1}(Q_j)}{\varphi_{\mathcal V,p}(Q_j)\log^2\left (\delta/\varphi_{\mathcal V,p}(Q_j)\right)}.$$
Then, there exists a positive constant $C_{\mathcal V,p}$ depending only in $\mathcal V,p$ and $Q_j\ (1\le j\le q)$ such that
$$ \sum_{j=1}^q\omega_j\Phi_{\mathcal V,jp}\ge C_{\mathcal V,p}\left (\prod_{j=1}^q\Phi_{\mathcal V,jp}^{\omega_j}\right)^{1/(M-p)} $$
holds on $S-\bigcup_{1\le j\le q}\{z;\varphi_{\mathcal V,p}(Q_j)(z)=0\}$.
\end{theorem}

\begin{proof}
Suppose that 
$$ [Q_j]=\sum_{i=0}^Ma_{ji}[v_i]\ (1\le j\le q).$$
We define $a_j:=(a_{j0},\ldots,a_{jM})\in\C^{M+1}$ and consider the set $\mathcal R_p$ of all subsets $R$ of $Q=\{1,2,\ldots, q\}$  such that $\rank_{\C}\{a_j; j\in R\}\le M-p$. For each $P\in G(M,p)$ (the Grassmannian manifold of all $(p+1)$-dimension linear subspaces of $\C^{M+1}$), take a decomposable $(p+1)$-vector $E$ such that
$$ P=\{X\in\C^{M+1}; E\wedge X=0\}$$
and set
$$ \phi_p(P)=\underset{R\in\mathcal R_p}{\max}\min\left\{\dfrac{\left|E\vee a_j\right|^2}{|E|^2};\ j\not\in R\right\}. $$
Then $\phi_p(P)$ depends only on $P$ and may be regarded as a function on the Grassmannian manifold $G(M,p)$. For each nonzero $(k+1)$-vector $E=E_0\wedge E_1\wedge\ldots\wedge E_k$ we set
$$R=\{j\in Q;\ E\vee a_j=0\}.$$
Note that, $E\vee a_j=0$ if and only if $a_j$ is contained in the orthogonal complement of the vector space $\mathrm{Span}(E_0,\ldots, E_k)$. Then we see
$$\rank_{\C}\{a_j; \ j\in R\}=\dim\mathrm{Span}(a_j;\ j\in R)\le M - p,$$
namely, $R\in\mathcal R_p$. Hence $\phi_p$ is positive everywhere on $G(M,p)$.
Since $\phi_p$ is obviously continuous and $G(M,p)$ is compact, we can take a positive constant $\delta$ such that $\phi_p(P)>\delta$ for each $P\in G(M,p)$.

Take a point $z$ with $F_{\mathcal V,p}(z)\ne 0$. The vector space generated
by $F^{(0)}_{\mathcal V}(z), F^{(1)}_{\mathcal V}(z)$,\ldots, $F^{(p)}_{\mathcal V}(z)$ determines a point in $G(M,p)$. Therefore, there is a set $R$ in $\mathcal R_p$ with 
$$ \rank_{\C}\{a_j; j\in R\}\le M-p $$ 
such that $\varphi_{\mathcal V,p}(Q_j)(z)\ge\delta$ for all $j\not\in R$. Then, we can choose a finite positive constant $K$ depending only on $\mathcal V,Q_j\ (1\le j\le q)$ such that $\Phi_{\mathcal V,jp}(z)\le K$ for all $j\not\in R$. Set
$$T:= \{j;\ \Phi_{\mathcal V,jp}(z)> K\}, l:= \sum_{j\in T}\omega_j.$$

We distinguish the following two cases.

\textit{Case 1: }Assume that $T$ is an empty set. We have
\begin{align*}
\sum_{j=1}^q\omega_j\Phi_{\mathcal V,jp}&\ge\left (\sum_{j=1}^q\omega_j\right)\left (\prod_{j=1}^q\Phi_{\mathcal V,jp}^{\omega_j}\right)^{\frac{1}{\sum_j\omega_j}}\\ 
& \ge (p+1)K\left (\prod_{j=1}^q\left(\dfrac{\Phi_{\mathcal V,jp}}{K}\right)^{\frac{\omega_j}{M-p}}\right)^{\frac{M-p}{\sum_{j}\omega_j}}\\
&\ge (p+1)K\left (\prod_{j=1}^q\left(\frac{\Phi_{\mathcal V,jp}}{K}\right)^{\omega_j}\right)^{\frac{1}{M-p}},
\end{align*}
because $\sum_{j=1}^q\omega_j=\tilde\omega (q-2N+k-1)+k+1\ge \dfrac{q(k+1)}{2N-k+1}\ge M+1>M-p$.

\textit{Case 2: }Assume that $T\ne\emptyset$. We see that $T\subset R$ and so $\rank_{\C}\{a_j;\ j\in T\}\le M-p$ holds. Since the intersection of $V$ and arbitrary $N+1$ hyperplanes $Q_j\ (1\le j\le q)$ is empty, there are at most $N$ indexes $j$ satisfying $\langle F^{(0)}_{\mathcal V}(z),a_j\rangle =0$ (i.e., $Q_j(F)(z)=0$). It yields that $\sharp T<N+1$, and hence $l\le \rank_{\C}\{a_j;\ j\in T\}\le M-p$. Then, we have
\begin{align*}
\sum_{j=1}^q\omega_j\Phi_{\mathcal V,jp}&\ge\sum_{j\in T}\omega_j\Phi_{\mathcal V,jp} \ge Kl\prod_{j\in T}\left(\frac{\Phi_{\mathcal V,jp}}{K}\right)^{\frac{\omega_j}{l}}\\
&\ge Kl\prod_{j\in T}\left(\frac{\Phi_{\mathcal V,jp}}{K}\right)^{\omega_j/(M-p)}\ge C\prod_{j=1}^q\left(\frac{\Phi_{\mathcal V,jp}}{K}\right)^{\omega_j/(M-p)},
\end{align*}
for some positive constant $C>0$, which depends only on $Q_1,\ldots,Q_q$.

From the above two cases, we get the conclusion of the theorem.
\end{proof}

\begin{theorem}\label{thm3.3}
Let the notations and the assumption be as in Theorem \ref{thm3.2} and let $\tilde\omega$ be the Nochka constant for these hypersurfaces (defined in the Lemma \ref{lem2.3}). Then, for every $\epsilon>0$, there exist positive numbers $\delta\ (>1)$ and $C$, depending only on $\mathcal V,\epsilon$ and $Q_j$ such that
\begin{align*}
dd^c&\log\dfrac{\prod_{p=0}^{M-1}|F_{\mathcal V,p}|^{2\epsilon}}{\prod_{1\le j\le q,0\le p\le M-1}\log^{2\omega_j}\left(\delta/\varphi_{\mathcal V,p}(Q_j)\right)}\\
&\ge C\left (\dfrac{|F_{\mathcal V,0}|^{2\left (\tilde\omega(q-(2N-k+1))-M+k\right)}|F_{\mathcal V,M}|^2}{\prod_{j=1}^q(|F_{\mathcal V,0}(Q_j)|^2\prod_{p=0}^{M-1}\log^2(\delta/\varphi_{\mathcal V,p}(Q_j)))^{\omega_j}}\right)^{\frac{2}{M(M+1)}}dd^c|z|^2.
\end{align*}
\end{theorem}

\begin{proof}
We denote by $A$ the left hand side. Thus
$$ A=\epsilon\sum_{p=0}^{M-1}\Omega_p+\sum_{j=1}^q \omega_j\sum_{p=0}^{M-1}dd^c\log\dfrac{1}{\log^{2}\left (\delta/\varphi_{\mathcal V,p}(Q_j)\right)}.$$
Choose a positive number $\delta_0(\epsilon/l)$ with the properties as in Proposition \ref{pro3.1}, where $l=\sum_{j=1}^q\omega_j$. For an arbitrarily fixed $\delta\ge \delta_0(\epsilon/l)$, we obtain
\begin{align*}
A&\ge\epsilon\sum_{p=0}^{M-1}\Omega_p+\sum_{j=1}^q\omega_j\sum_{p=0}^{M-1}\left (\dfrac{2\varphi_{\mathcal V,p+1}(Q_j)}{\varphi_{\mathcal V,p}(Q_j)\log^2(\delta/\varphi_{\mathcal V,p}(Q_j))}-\dfrac{\epsilon}{l}\right)\Omega_p\\ 
& =\sum_{p=0}^{M-1}2\left (\sum_{j=1}^q\omega_j\Phi_{\mathcal V,jp}\right).
\end{align*}
From Theorem \ref{thm3.2}, it follows that
$$ A\ge C_1\sum_{p=0}^{M-1}2\left (\prod_{j=1}^q\Phi_{\mathcal V,jp}^{\omega_j}\right)^{\frac{1}{M-p}}\Omega_p, $$
for some positive constant $C_1>0$. Let $\Omega_p=h_pdd^c|z|^2$, we have
\begin{align*}
A&\ge 2C_1\sum_{p=0}^{M-1}\left (h_p^{M-p}\prod_{j=1}^q\Phi_{\mathcal V,jp}^{\omega_j}\right)^{\frac{1}{M-p}}dd^c|z|^2\\ 
& \ge C'_1\sum_{p=0}^{M-1}(M-p)\left (h_p^{M-p}\prod_{j=1}^q\Phi_{\mathcal V,jp}^{\omega_j}\right)^{\frac{1}{M-p}}dd^c|z|^2\\ 
&\ge C_2\prod_{p=0}^{M-1}\left (h_p^{M-p}\prod_{j=1}^q\Phi_{\mathcal V,jp}^{\omega_j}\right)^{\frac{2}{M(M+1)}}dd^c|z|^2,
\end{align*}
for some positive constants $C'_1,C_2$. On the other hand, we note that
\begin{align*}
\prod_{p=0}^{M-1}\Phi_{\mathcal V,jp}&=\prod_{p=0}^{M-1}\dfrac{\varphi_{\mathcal V,p+1}(Q_j)}{\varphi_{\mathcal V,p}(Q_j)}\dfrac{1}{\log^2(\delta/\varphi_{\mathcal V,p}(Q_j))}\\ 
& =\dfrac{|F_{0}|^2}{|F_{0}(Q_j)|^2}\prod_{p=0}^{M-1}\dfrac{1}{\log^2(\delta/\varphi_{\mathcal V,p}(Q_j))}
\end{align*}
and
$$ \prod_{p=0}^{M-1}h_p^{M-p}=\prod_{p=0}^{M-1}\left (\dfrac{|F_{\mathcal V,p-1}|^2|F_{\mathcal V,p+1}|^2}{|F_{\mathcal V,p}|^4}\right)^{M-p}=\dfrac{|F_{\mathcal V,M}|^2}{|F_{\mathcal V,0}|^{2(M+1)}},$$
because $\varphi_{\mathcal V,0}(Q_j)=|F_{\mathcal V,0}(Q_j)|/|F_{\mathcal V,0}|,\varphi_{\mathcal V,M}(Q_j)=1$. Therefore, we get
$$ A\ge C\left (\dfrac{|F_{\mathcal V,0}|^{2(l-M-1)}|F_{\mathcal V,M}|^2}{\prod_{j=1}^q(|F_{\mathcal V,0}(Q_j)|^2\prod_{p=0}^{M-1}\log^2(\delta/\varphi_{\mathcal V,p}(Q_j)))^{\omega_j}}\right)^{\frac{2}{M(M+1)}}dd^c|z|^2.$$
Since $l-M-1=\tilde\omega (q-2N+k-1)-M+k$, we have the conclusion of the theorem.
\end{proof}

Applying Theorems \ref{thm3.2} and \ref{thm3.3}, we now give the following two main lemmas of this paper (Lemmas \ref{lem3.7} and \ref{lem3.8}) as follows. 
\begin{lemma}\label{lem3.7} 
Let $V,\mathcal V, d, M,\{Q_i\}_{i=1}^q$ and $\{\omega_i\}_{i=1}^q$ be as in Theorem \ref{thm3.2}. Let $f:\Delta (R)\rightarrow V\subset\P^n(\C)$ be a holomorphic map with a reduced representation $F=(f_0,\ldots,f_n)$, which is nondegenerate over $I_d(V)$. Assume that there are positive real numbers $\eta_j$, functions $u_j$ satisfying two conditions $(D_1),(D_2)$ with respect to $F,Q_j$ and $\eta_j\ (1\le j\le q)$. Then for an arbitrarily given $\epsilon$ satisfying 
$$\gamma=\sum_{j=1}^q\omega_j (1-\eta_j)-M-1>\epsilon(\sigma_{M+1}+1),$$ 
the pseudo-metric $d\tau^2=\eta^2|dz|^2$, where
$$ \eta=\left( \dfrac{|F_{\mathcal V,0}|^{\gamma-\epsilon(\sigma_{M+1}+1)}\prod_{j=1}^qe^{\omega_ju_j}|F_{\mathcal V,M}|\prod_{p=0}^{M}|F_{\mathcal V,p}|^{\epsilon}}{\prod_{j=1}^q|Q_j(F)|^{\omega_j-\frac{\epsilon}{q}}\cdot(\prod_{j,p}\log (\delta/\varphi_{\mathcal V,p}(Q_j)))^{\omega_j}}\right )^{\frac{1}{\sigma_M+\epsilon\tau_M}}$$
and $\delta$ is the number satisfying the conclusion of Theorem \ref{thm3.3}, is continuous and has strictly negative curvature.
\end{lemma}
\begin{proof}
We see that the function $\eta$ is continuous at every point $z$ with $\prod_{j=1}^qQ_j(F)(z)\ne 0$. For a point $z_0\in S$ such that $\prod_{j=1}^qQ_j(F)(z_0)= 0$, from Theorem \ref{thm3.5} and the property of $u_j$, we have
\begin{align*}
\nu_{\eta}(z_0)&\ge \frac{1}{\sigma_M+\epsilon\tau_M}\biggl (\nu_{F_{\mathcal V,M}}(z_0))+\sum_{j=1}^q\min(\nu_{Q_j(F)}(z_0),M)-\sum_{j=1}^q\omega_j\nu_{Q_j(F)}(z_0)\\
&+\sum_{j=1}^q\omega_j\nu_{e^{u_j}|z-z_0|^{-\min(M,\nu_{Q_j(F)}(z_0))}}(z_0)\biggl)\ge 0.
\end{align*}
  This implies that $d\tau^2$ is a continuous pseudo-metric on $\Delta(R)$. 

We now prove that $d\tau^2$ has strictly negative curvature on $\Delta$. Let $\Omega$ be the Fubini-Study form of $\P^n(\C)$ and denote by $\Omega_f$ the pull-back of $\Omega$ by $f$. By Theorem \ref{thm3.5}, we have
$$dd^c\log\dfrac{|F_{\mathcal V,M}|}{\prod_{j=1}^q|Q_j(F)|^{\omega_j-\frac{\epsilon}{q}}}+\sum_{j=1}^q\omega_jdd^c[u_j]\ge 0.$$
Then by Theorems \ref{thm3.4} and \ref{thm3.3} and the definition of $\eta$, we have
\begin{align}\label{new3}
\begin{split}
dd^c\log\eta&\ge\dfrac{\gamma-\epsilon(\sigma_{M+1}+1)}{\sigma_M+\epsilon\tau_M}d\Omega_f+\dfrac{\epsilon}{2(\sigma_M+\epsilon\tau_M)}dd^c\log\left(|F_{\mathcal V,0}|\cdots|F_{\mathcal V,M}|\right)\\
& +\dfrac{1}{2(\sigma_M+\epsilon\tau_M)}dd^c\log\dfrac{\prod_{p=0}^{M-1}|F_{\mathcal V,p}|^{2\epsilon}}{\prod_{p=0}^{M-1}\log^{4\omega_j}(\delta/\varphi_{\mathcal V,p}(Q_j))}\\
&\ge\dfrac{\epsilon\tau_M}{2\sigma_M(\sigma_M+\epsilon\tau_M)}\left(\dfrac{|F_{\mathcal V,0}|^2\cdots |F_{\mathcal V,M}|^2}{|F_{\mathcal V,0}|^{2\sigma_{M+1}}}\right)^{\frac{1}{\tau_M}}dd^c|z|^2\\
&+C_0\left (\dfrac{|F_{\mathcal V,0}|^{2\left(\tilde\omega (q-2N+k-1)-M+k\right)}|F_{\mathcal V,M}|^2}{\prod_{j=1}^q(|Q_j(F)|^2\prod_{p=0}^{M-1}\log^2(\delta/\varphi_{\mathcal V,p}(Q_j)))^{\omega_j}}\right)^{\frac{1}{\sigma_M}}dd^c|z|^2\\
&\ge C_1\left (\dfrac{|F_{\mathcal V,0}|^{\tilde\omega (q-2N+k-1)-M+k-\epsilon\sigma_{M+1}}|F_{\mathcal V,M}|\prod_{p=0}^M|F_{\mathcal V,p}|^\epsilon}{\prod_{j=1}^q(|Q_j(F)|\prod_{p=0}^{M-1}\log(\delta/\varphi_{\mathcal V,p}(Q_j)))^{\omega_j}}\right)^{\frac{2}{\sigma_M+\epsilon\tau_M}}dd^c|z|^2
\end{split}
\end{align}
for some positive constants $C_0,C_1$, where the last inequality comes from the H\"{o}lder's inequality.
On the other hand, we have $e^{u_j}\le \|F\|^{d\eta_j}$, $|Q_j(F)|\le \|F\|^{d\eta_j}$, and
\begin{align*}
&|F_{\mathcal V,0}|^{\tilde\omega (q-2N+k-1)-M+k-\epsilon\sigma_{M+1}}=|F_{\mathcal V,0}|^{\gamma-\epsilon\sigma_{M+1}+\sum_{j=1}^q\omega_j\eta_j}\\
&\hspace{20pt}\ge |F_{\mathcal V,0}|^{\gamma-\epsilon(\sigma_{M+1}+1)}e^{\omega_1u_1}|Q_1(F)|^{\frac{\epsilon}{q}}\cdots e^{\omega_qu_q}|Q_q(F)|^{\frac{\epsilon}{q}}.
\end{align*}
This implies that $\Delta \log\eta^2\ge C_2\eta^2$ for some positive constant $C_2$. Therefore, $d\tau^2$ has strictly negative curvature.
\end{proof}

\begin{lemma}\label{lem3.8}
Let $V,\mathcal V, d, M,\{Q_i\}_{i=1}^q$ and $\{\omega_i\}_{i=1}^q$ be as in Theorem \ref{thm3.2}. Let $f:\Delta (R)\rightarrow V\subset\P^n(\C)$ be a holomorphic map with a reduced representation $F=(f_0,\ldots,f_n)$, which is nondegenerate over $I_d(V)$. Assume that there are positive real numbers $\eta_j$, functions $u_j$ satisfying two conditions $(D_1),(D_2)$ with respect to $F,Q_j$ and $\eta_j\ (1\le j\le q)$. Then for every $\epsilon >0$ satisfying 
$$\gamma=\sum_{j=1}^q\omega_j (1-\eta_j)-M-1>\epsilon(\sigma_{M+1}+1),$$ 
there exists a positive constant $C$, depending only on $\mathcal V, Q_j,\eta_j,\omega_j\ (1\le j\le q)$, such that
$$ \eta=\dfrac{|F_{\mathcal V,0}|^{\gamma-\epsilon(\sigma_{M+1}+1)}\prod_{j=1}^qe^{\omega_ju_j}|F_{\mathcal V,M}|^{1+\epsilon}\prod_{j,p}|F_{\mathcal V,p}(Q_j)|^{\epsilon/q}}{\prod_{j=1}^q|Q_j(F)|^{\omega_j-\frac{\epsilon}{q}}}\le C\left(\dfrac{2R}{R^2-|z|^2}\right)^{\sigma_M+\epsilon\tau_M}.$$
\end{lemma}
\begin{proof}
As in the proof of Lemma \ref{lem3.7}, we have
$$dd^c\log\eta\le C_2\eta^2dd^c|z|^2.$$
According to Lemma \ref{lem3.6}, this implies that
$$ \eta\le C_3\dfrac{2R}{R^2-|z|^2},$$
for some positive constant $C_3$. Then we have
$$ \left( \dfrac{|F_{\mathcal V,0}|^{\gamma-\epsilon(\sigma_{M+1}+1)}\prod_{j=1}^qe^{\omega_ju_j}|F_{\mathcal V,M}|\prod_{p=0}^{M}|F_{\mathcal V,p}|^{\epsilon}}{\prod_{j=1}^q(|Q_j(F)|^{\omega_j-\frac{\epsilon}{q}}\cdot(\prod_{j,p}\log (\delta/\varphi_{\mathcal V,p}(Q_j)))^{\omega_j}}\right )^{\frac{1}{\sigma_M+\epsilon\tau_M}}\le C_3\dfrac{2R}{R^2-|z|^2}.$$
It follows that
$$ \left( \dfrac{|F_{\mathcal V,0}|^{\gamma-\epsilon(\sigma_{M+1}+1)}\prod_{j=1}^qe^{\omega_ju_j}|F_{\mathcal V,M}|^{1+\epsilon}\prod_{j,p}|F_{\mathcal V,p}(Q_j)|^{\frac{\epsilon}{q}}}{\prod_{j=1}^q|Q_j(F)|^{\omega_i-\frac{\epsilon}{q}}\prod_{j=1}^q\prod_{p=0}^{M-1}(\varphi_{\mathcal V,p}(Q_j))^{\frac{\epsilon}{2q}}(\log (\delta/\varphi_{\mathcal V,p}(Q_j)))^{\omega_j}}\right )^{\frac{1}{\sigma_M+\epsilon\tau_M}}\le C_3\dfrac{2R}{R^2-|z|^2}.$$
Note that the function $x^{\frac{\epsilon}{q}}\log^{\omega}\left (\dfrac{\delta}{x^2}\right)\ (\omega>0,0<x\le 1)$ is bounded. Then we have
$$ \left( \dfrac{|F_{\mathcal V,0}|^{\gamma-\epsilon(\sigma_{M+1}+1)}\prod_{j=1}^qe^{\omega_ju_j}|F_{\mathcal V,M}|^{1+\epsilon}\prod_{j,p}|F_{\mathcal V,p}(Q_j)|^{\frac{\epsilon}{q}}}{\prod_{j=1}^q|Q_j(F)|^{\omega_i-\frac{\epsilon}{q}}}\right )^{\frac{1}{\sigma_M+\epsilon\tau_M}}\le C_4\dfrac{2R}{R^2-|z|^2},$$
for a positive constant $C_4$. The lemma is proved.
\end{proof}

\section{Modified defect relation for Gauss map with hypersurfaces}

In this section, we will prove the main theorem of the paper. Firstly, we need some following preparation.

\begin{lemma}[{cf. \cite[Lemma 1.6.7]{Fu93}}]\label{lem4.1}
Let $d\sigma^2$ be a conformal flat metric on an open Riemann surface $S$. Then for every point $p\in S$, there is a holomorphic and locally biholomorphic map $\Phi$ of a disk $\Delta (R_0):=\{w:|w| <R_0\}\ (0 <R_0\le\infty)$ onto an open neighborhood of $p$ with $\Phi(0)=p$ such that $\Phi$ is a local isometric, namely the pull-back $\Phi^*(d\sigma^2)$ is equal to the standard (flat) metric on $\Delta(R_0)$, and for some point $a_0$ with $|a_0|=1$, the curve $\Phi (\overline{0,R_0a_0})$ is divergent in $S$ (i.e., for any compact set $K\subset S$, there exists an $s_0<R_0$ such that $\Phi (\overline{0,s_0a_0})$ does not intersect $K$).
\end{lemma}

Let $x=(x_0,\ldots,x_{n}): S\rightarrow\R^{m}$ be the immersion of the minimal surface $S$ into $\R^{m}$. Let $(x,y)$ be an isothermal coordinate of $S$. Then $z=x+iy$ is a holomorphic coordinate of $S$. The generalized Gauss map $g$ of $x$ is defined by
$$ g:S\rightarrow\P^{n}(\C), g:=\left(\dfrac{\partial x_0}{\partial z}:\cdots:\dfrac{\partial x_{n}}{\partial z}\right),\ \ (n=m-1).$$
Denote by $G$ the reduced representation of $g$ defined by
$$G=\left(\dfrac{\partial x_0}{\partial z},\ldots,\dfrac{\partial x_{n}}{\partial z}\right).$$
Also, the metric $ds^2$ on $S$ induced by the canonical metric on $\R^{m}$ satisfies
$$ ds^2=2\|G\|^2|dz|^2.$$
We note that, $g$ is holomorphic since $S$ is minimal.

\begin{proof}[Proof of Theorem \ref{1.1}]
Let $A$ be an annular end of $S$, that is, $A=\{z;\ 0 < 1/r \le  |z| < r < \infty\}$, where $z$ is the conformal coordinate. Denote by $G$  the reduced representation of $g$ defined with respect to the coordinate $z$.

We fix a $\C$-ordered basis $\mathcal V=([v_0],\ldots,[v_M])$ of $I_d(V)$ as in the Section 3. By replacing $Q_i$ with $Q_i^{d/\deg Q_i}\ (1\le i\le q)$ if necessary, we may assume that all $Q_i\ (1\le i\le q)$ are of the same degree $d$. Suppose that
$$ [Q_j]=\sum_{i=0}^Ma_{ji}[v_i],$$
where $\sum_{i=0}^M|a_{ji}|^2=1$.

Since $g$ is nondegenerate over $I_d(V)$, the contact function satisfying
$$ G_{\mathcal V,p}(Q_j)\not\equiv 0, \forall 1\le j\le q,0\le p\le M.$$
Then, for each $j,p\ (1\le j\le q,0\le p\le M)$, we may choose $i_1,\ldots,i_p$ with $0\le i_1<\cdots<i_p\le M$ such that
$$ \psi(G)_{jp}=\sum_{l\ne i_1,\ldots,i_p}a_{jl}W(v_l(G),v_{i_1}(G_z),\ldots,v_{i_p}(G))\not\equiv 0.$$
We also note that $\psi(G)_{j0}=G_{\mathcal V,0}(Q_j)=Q_j(G)$ and $\psi(G)_{jM}=G_{\mathcal V,M}$.

Suppose contrarily that
$$\sum_{j=1}^q\delta^{H,M}_{g,A}(Q_j)>\dfrac{(2N-k+1)(M+1)}{k+1}+\dfrac{(2N-k+1)M(M+1)}{2d(k+1)}.$$
Then, there exist numbers $\eta_j>0\ (1\le j\le q)$ and functions $u_j$ satisfying two conditions $(D_1),(D_2)$ with respect to $G$ and $\{Q_j\}_{j=1}^q$ such that
$$ \sum_{j=1}^q(1-\eta_j)> \dfrac{(2N-k+1)(M+1)(M+2d)}{2d(k+1)}.$$
From Theorem \ref{lem2.3}, we have
$$ (q-2N+k-1)\tilde\omega=\sum_{j=1}^q\omega_j-k-1;\ \tilde\omega\ge\omega_j>0 \text{ and }\tilde\omega\ge\dfrac{k+1}{2N-k+1}.$$
Therefore,
\begin{align}\label{new2}
\begin{split}
\sum_{j=1}^q\omega_j(1-\eta_j)-M-1&\ge\tilde\omega(q-2N+k-1-\sum_{j=1}^q\eta_j)-M+k\\ 
&\ge\dfrac{k+1}{2N-k+1}\left (\sum_{j=1}^q(1-\eta_j)-2N+k-1\right)-M+k\\
&= \dfrac{k+1}{2N-k+1}\left (\sum_{j=1}^q(1-\eta_j)-\dfrac{(2N+k-1)(M+1)}{k+1}\right)\\
&>\dfrac{k+1}{2N-k+1}\cdot\dfrac{(2N+k-1)M(M+1)}{2d(k+1)}=\dfrac{\sigma_M}{d}.
\end{split}
\end{align}
Then, we can choose a rational number $\epsilon\ (>0)$ such that
$$ \dfrac{d(\sum_{j=1}^q\omega_j(1-\eta_j)-M-1)-\sigma_M}{d(\sigma_{M+1}+1)+\tau_M}>\epsilon> \dfrac{d(\sum_{j=1}^q\omega_j(1-\eta_j)-M-1)-\sigma_M}{\frac{1}{q}+d(\sigma_{M+1}+1)+\tau_M}.$$
We define the following numbers
\begin{align*}
h&:=d(\sum_{j=1}^q\omega_j(1-\eta_j)-M-1)-\epsilon d(\sigma_{M+1}+1)>\sigma_M+\epsilon\tau_{M},\\ 
\rho&:=\dfrac{1}{h}(\sigma_M+\epsilon\tau_M),\\
\rho^*&:=\dfrac{1}{(1-\rho)h}=\dfrac{1}{d(\sum_{j=1}^q\omega_j(1-\eta_j)-M-1)-\sigma_M-\epsilon(d(\sigma_{M+1}+1)+\tau_M)}. 
\end{align*}
It is clear that $0<\rho<1$ and $\frac{\epsilon\rho^*}{q}>1.$

We consider a set
$$ A_1=\{a\in A; \psi (G)_{jp}(a)\ne 0,u_j(a)\ne -\infty\ \forall j=1,\ldots,q; p=0,\ldots,M\}$$
and define a new pseudo-metric on $A_1$ as follows
$$ d\tau^2=\left (\dfrac{\prod_{j=1}^q\|Q_j(G)\|^{\omega_j-\frac{\epsilon}{q}}}{|G_{\mathcal V,M}|^{1+\epsilon}\prod_{j=1}^qe^{\omega_ju_j}\prod_{j,p}|\psi(G)_{jp}|^{\frac{\epsilon}{q}}}\right)^{2\rho^*}|dz|^2.$$
Here and throughout this section, for simplicity we will write $\prod_{j,p}$ for $\prod_{j=1}^q\prod_{p=1}^{M-1}$.

Since $Q_j(G),G_{\mathcal V,M},\psi(G)_{jp}$ are all holomorphic functions and $u_j\ (1\le j\le q)$ are harmonic functions on $A_1$, $d\tau$ is flat on $A_1$. 

\begin{claim}\label{cl4.2}
$d\tau$ is complete on the set $\{z;\ |z|=r\}\cup\{z | \psi(G)_{jp}(z)=0\}\cup\{z |\ Q_j(G)(z)=0\}$, i.e., set $\{z;\ |z|=r\}\cup\{z |\ \psi(G)_{jp}(z)=0\}\cup\{z |\ Q_j(G)(z)=0\}$ is at infinite distance from any interior point.
\end{claim}
Indeed, if $a\in A$ such that $\prod_{j=1}^q\left ( Q_j(G)(a)\cdot\prod_{p=0}^M\psi(G)_{jp}(a)\right )=0$ then by Theorem \ref{thm3.5}, we have
\begin{align*}
\nu_{d\tau}(a)&\le -\biggl [\nu_{G_{\mathcal V,M}}(a)-\sum_{j=1}^q\omega_j\nu_{Q_j(G)}(a)+\sum_{j=1}^q\omega_j\min\{\nu_{Q_j(G)}(a),M\}\\ 
&\ \ +\sum_{j=1}^q\omega_j\nu_{e^{u_j|z-z_0|^{-\min\{\nu_{Q_j(G)}(a),M\}}}}(a)+\epsilon\nu_{G_{\mathcal V,M}}(a)\\
&\ \ +\left(\dfrac{\epsilon}{q}\sum_{j=1}^q\big(\min\{\nu_{Q_j(G)}(a),M\}+\sum_{p=0}^{M-1}\nu_{\psi(G)_{jp}}(a)\big)\right)\biggl ]\rho^*\\
&\le -\epsilon\rho^*\nu_{G_{\mathcal V,M}}(a)-\dfrac{\epsilon\rho^*}{q}\sum_{j=1}^q\big(\min\{\nu_{Q_j(G)}(a),M\}+\sum_{p=0}^{M-1}\nu_{\psi(G)_{jp}}(a)\big)\le -\dfrac{\epsilon\rho^*}{q}
\end{align*}
(since $\min\{\nu_{Q_j(G)}(a),M\}>1$ on $\supp\min\{\nu_{Q_j(G)}(a),M\}$). Then, there is a positive constant $C$ such that
$$ |d\tau|\ge\dfrac{C}{|z-z(a)|^{\frac{\epsilon\rho^*}{q}}}|dz|$$
in a neighborhood of $a$, where $\frac{\epsilon\rho^*}{q}>1$. Then $d\tau$ is complete on the set
$$\{z | \psi(G)_{jp}(z)=0\}\cup\{z |\ Q_j(G)(z)=0\}.$$

Now suppose that $d\tau^2$ is not complete on the set $\{z;\ |z|=r\}$. Then, there exists a curve $\gamma:[0,1)\rightarrow A_1$, where $\gamma (1)\in \{z;\ |z|=r\}$ so that the length of $\gamma$ in the metric $d\tau$ satisfying $l_{d\tau}(\gamma)<\infty$. Furthermore, we may also assume that $dist(\gamma (0), \{z;\ |z|=1/r\}) > 3l_{d\tau}(\gamma)$.

Take a disk $\Delta$ (in the metric induced by $d\tau$) with centered at $\gamma(0)$. Since $d\tau$ is flat, by Lemma \ref{lem4.1}, $\Delta$ is isometric to an ordinary disk in the plane. Let $\Phi:\Delta(r)=\{|\omega|<r\}\rightarrow\Delta$ be this isometric with $\Phi(0)=\gamma(0)$. Extend $\Phi$ as a local isometric into $A_1$ to a the largest disk possible $\Delta(R)=\{|\omega|<R\}$, and denoted again by $\Phi$ this extension (for simplicity, we may consider $\Phi$ as the exponential map).  Since $\int_{\gamma|_{[0,1)}}d\tau<{\rm dist}(\gamma (0), \{z;\ |z|=1/r\})/3$, we have $R\le {\rm dist}(\gamma (0), \{z;\ |z|=1/r\})/3$. Hence, the image of $\Phi$ is bounded away from $\{z;\ |z|=1/r\}$ by a distance at least ${\rm dist}(\gamma (0), \{z;\ |z|=1/r\})/3$. Since $\Phi$ cannot be extended to a larger disk, it must be hold that the image of $\Phi$ goes to the boundary $A_1$. But, this image cannot go to points $a$ of $A$ with $\prod_{j=1}^q\left(Q_j(G)(a)\prod_{p=0}^M\psi(G)_{jp}(a)\right)=0$, since we have already shown that $\gamma(0)$ is infinitely far away in the metric with respect to these points. Then the image of $\Phi$ must go to the set $\{z;\ |z|=r\}$. Hence, by again Lemma \ref{lem4.1}, there exists a point $w_0$ with $|w_0|= R$ so that $\Gamma=\Phi(\overline{0,w_0})$ is a divergent curve on $S$.

We now show that $\Gamma$ has finite length in the original metric $ds^2$ on $S$, contradicting the completeness of $S$. Let $f:=g\circ\Phi :\Delta(R)\rightarrow V\subset\P^n(\C)$, $\omega\in\Delta (R)\mapsto f(\omega)\in V$, be a holomorphic curve which is nondegenerate over $I_d(V)$. Then $f$ have a reduced representation
$$ F=(f_0,\ldots,f_n),$$
where $f_i=G\circ\Phi\ (0\le i\le n).$
Hence, we have:
\begin{align*}
\Phi^*ds^2&=2\|G\circ\Phi\|^2|\Phi^*dz|^2=2\|F\|^2\left|\dfrac{d(z\circ\Phi)}{dw}\right|^2|dw|^2,\\
F_{\mathcal V,M}&=(G\circ\Phi)_{\mathcal V,M}=G_{\mathcal V,M}\circ\Phi\cdot\left(\dfrac{d(z\circ\Phi)}{dw}\right)^{\sigma_M},\\
\psi(F)_{jp}&=\psi(G\circ\Phi)_{jp}=\psi(G)_{jp}\cdot\left(\dfrac{d(z\circ\Phi)}{dw}\right)^{\sigma_p}, (0\le p\le M).
\end{align*}
On the other hand, since $\Phi$ is locally isometric,
\begin{align*}
&|dw|=|\Phi^*d\tau|\\
&=\left (\dfrac{\prod_{j=1}^q|Q_j(G)\circ\Phi|^{\omega_j-\frac{\epsilon}{q}}}{|G_{\mathcal V,M}\circ\Phi|^{1+\epsilon}\prod_{j=1}^qe^{\omega_ju_j\circ\Phi}\prod_{j,p}|\psi(G)_{jp}\circ\Phi|^{\epsilon/q}}\right)^{\rho^*}\left|\dfrac{d(z\circ\Phi)}{dw}\right|\cdot|dw|\\
&=\left (\dfrac{\prod_{j=1}^q|Q_j(F)|^{\omega_j-\frac{\epsilon}{q}}}{|F_{\mathcal V,M}|^{1+\epsilon}\prod_{j=1}^qe^{\omega_ju_j\circ\Phi}\prod_{j,p}|\psi(F)_{jp}|^{\epsilon/q}}\right)^{\rho^*}\left|\dfrac{d(z\circ\Phi)}{dw}\right|^{1+h\rho\rho^*}\cdot|dw|
\end{align*}
(because $1+\rho^*(\sigma_M+\epsilon\tau_M)=1+h\rho\rho^*$).
This implies that
\begin{align*}
\left|\dfrac{d(z\circ\Phi)}{dw}\right|&=\left (\dfrac{|F_{\mathcal V,M}|^{1+\epsilon}\prod_{j=1}^qe^{\omega_ju_j\circ\Phi}\prod_{j,p}|\psi(F)_{jp}|^{\epsilon/q}}{\prod_{j=1}^q|Q_j(F)|^{\omega_j-\frac{\epsilon}{q}}}\right)^{\frac{\rho^*}{1+h\rho\rho^*}}\\
&\le \left (\dfrac{|F_{\mathcal V,M}|^{1+\epsilon}\prod_{j=1}^qe^{\omega_ju_j\circ\Phi}\prod_{j,p}|F_{\mathcal V,p}(Q_j)|^{\epsilon/q}}{\prod_{j=1}^q|Q_j(F)|^{\omega_j-\frac{\epsilon}{q}}}\right)^{\frac{\rho^*}{1+h\rho\rho^*}}\\
&=\left (\dfrac{|F_{\mathcal V,M}|^{1+\epsilon}\prod_{j=1}^qe^{\omega_ju_j\circ\Phi}\prod_{j,p}|F_{\mathcal V,p}(Q_j)|^{\epsilon/q}}{\prod_{j=1}^q|Q_j(F)|^{\omega_j-\frac{\epsilon}{q}}}\right)^{1/h}.
\end{align*}
Hence, we have
\begin{align*}
\Phi^*ds&\le\sqrt{2}\|F\|\left (\dfrac{|F_{\mathcal V,M}|^{1+\epsilon}\prod_{j=1}^qe^{\omega_ju_j\circ\Phi}\prod_{j,p}|F_{\mathcal V,p}(Q_j)|^{\epsilon/q}}{\prod_{j=1}^q|Q_j(F)|^{\omega_j-\frac{\epsilon}{q}}}\right)^{1/h}|dw|\\
&=\sqrt{2}\left (\dfrac{|F_{\mathcal V,0}|^{\frac{h}{d}}|F_{\mathcal V,M}|^{1+\epsilon}\prod_{j=1}^qe^{\omega_ju_j\circ\Phi}\prod_{j,p}|F_{\mathcal V,p}(Q_j)|^{\epsilon/q}}{\prod_{j=1}^q|Q_j(F)|^{\omega_j-\frac{\epsilon}{q}}}\right)^{1/h}|dw|.
\end{align*}
Here, note that $\frac{h}{d}=\sum_{j=1}^q\omega_j(1-\eta_j)-M-1-\epsilon(\sigma_{M+1}+1)$. Then the inequality (\ref{new2}) yields that the conditions of Lemma \ref{lem3.6} are satisfied. Applying Lemma \ref{lem3.6} we have
$$ \Phi^*ds\le C\left (\dfrac{2R}{R^2-|w|^2}\right)^\rho|dw|,$$
for some positive constant $C$. Also, since $0<\rho<1$,
$$l_{ds^2}(\Gamma)\le\int_{\Gamma}ds=\int_{\overline{0,w_0}}\Phi^*ds\le C\cdot\int_{0}^R\left(\dfrac{2R}{R^2-|w|^2}\right)^{\rho}|dw|<+\infty. $$
This contradicts the assumption of the completeness of $S$ with respect to $ds^2$. Thus, Claim \ref{cl4.2} is proved.

Now, we set $d\tilde\tau^2=\lambda|dz|^2$ where
\begin{align*}
\lambda(z)&=\left (\dfrac{\prod_{j=1}^q\|Q_j(G)(z)\|^{\omega_j-\frac{\epsilon}{q}}}{|G_{\mathcal V,M}(z)|^{1+\epsilon}\prod_{j=1}^qe^{\omega_ju_j(z)}\prod_{j,p}|\psi(G)_{jp}(z)|^{\frac{\epsilon}{q}}}\right)^{2\rho^*}\\ 
& \times\left (\dfrac{\prod_{j=1}^q\|Q_j(G)(\frac{1}{z})\|^{\omega_j-\frac{\epsilon}{q}}}{|G_{\mathcal V,M}(\frac{1}{z})|^{1+\epsilon}\prod_{j=1}^qe^{\omega_ju_j(\frac{1}{z})}\prod_{j,p}|\psi(G)_{jp}(\frac{1}{z})|^{\frac{\epsilon}{q}}}\right)^{2\rho^*}.
\end{align*}
Also, we note that $\{u_j=-\infty\}=Q_j(G)^{-1}\{0\}$. This yields that the function 
$$\dfrac{\prod_{j=1}^q\|Q_j(G)(z)\|^{\omega_j-\frac{\epsilon}{q}}}{|G_{\mathcal V,M}(z))|^{1+\epsilon}\prod_{j=1}^qe^{\omega_ju_j(z)}\prod_{j,p}|\psi(G)_{jp}(z)|^{\frac{\epsilon}{q}}}$$ 
is continuous and non-vanishing on $A_1$. Similarly the function 
$$\dfrac{\prod_{j=1}^q\|Q_j(G)(\frac{1}{z})\|^{\omega_j-\frac{\epsilon}{q}}}{|G_{\mathcal V,M}(\frac{1}{z}))|^{1+\epsilon}\prod_{j=1}^qe^{\omega_ju_j(\frac{1}{z})}\prod_{j,p}|\psi(G)_{jp}(\frac{1}{z})|^{\frac{\epsilon}{q}}}$$
 is continuous and non-vanishing on $A_1'$, where
$$ A_1'=\left\{a\in A; \psi (G_z)_{jp}\left(\frac{1}{a}\right)\ne 0,u_j\left(\frac{1}{a}\right)\ne -\infty\ \forall j=1,\ldots,q; p=0,\ldots,M\right\}.$$
We set $ A'= A_1\cap A_1'$. These facts yield that the metric $d\tilde\tau^2$ is complete and flat on $A'$. 

From Lemma \ref{lem4.1}, there exist a local isometric $\Phi:\Delta(R_0)\rightarrow A'$. Since $(A',d\tilde\tau^2)$ is a complete Riemann surface, it is necessary that $R_0=\infty$. So $\Phi:\C\rightarrow A'\subset\{z;\ |z|<r\}$
is a non-constant holomorphic map, which contradicts to Liouville's theorem. Then, the supposition was wrong. Hence, Theorem \ref{1.1} is proved.
\end{proof}

\end{document}